\documentclass[12pt]{amsart}
\usepackage{amsmath}%
\usepackage{amsfonts}%
\usepackage{rotating}
\usepackage[x11names]{xcolor}
\usepackage{xspace}

\usepackage[mathscr]{eucal}

\newtheorem{theorem}{Theorem}[section]
\theoremstyle{plain}

\newtheorem{corollary}[theorem]{Corollary}

\newtheorem{definition}[theorem]{Definition}

\newtheorem{lemma}[theorem]{Lemma}
\newtheorem{proposition}[theorem]{Proposition}

\def \endit{\quad\vrule height 1.9ex width 1.174ex depth .2ex \par} 
\renewenvironment {proof}{{\par\bf Proof:{\enspace}}}{\endit\smallskip}
\everymath{\displaystyle}
\DeclareMathOperator{\Spann}{span}

\begin{document}
\title[Hopf algebra isomorphisms]{Isomorphisms and automorphisms of discrete multiplier Hopf C*-algebras}

\author{Dan Z. Ku\v{c}erovsk\'{y}}
\address
{UNB--F \newline%
\indent Canada\qquad E3B 5A3}%
\email{dkucerov@unb.ca}%
\thanks{We thank NSERC for financial support. Authour's e-mail address: \texttt{dkucerov@unb.ca}. Phone: ++1-506-458-7364}
\dedicatory{}
\begin{abstract} We construct Hopf algebra isomorphisms of discrete multiplier Hopf C*-algebras, and Hopf AF C*-algebras (generalized quantum UHF algebras), from $K$-theoretical data.  Some of the intermediate results are of independent interest, such as a result that Jordan maps of Hopf algebras intertwine antipodes, and the applications to automorphisms of Hopf algebras.
 \end{abstract}\maketitle
\newcommand{\Mult}{{\mathcal M}} 
\newcommand{\Span}{\displaystyle \Spann}
\newcommand{\compose}{\circ}
\newcommand{\Hgns}{\ensuremath{{\mathcal H}_{\text{\tiny GNS}}}} 
\renewcommand{\H}{\ensuremath{{\mathcal H}}} 
\newcommand{\C}{\mathbb C}
\newcommand{\N}{\mathbb N}
\newcommand{\Z}{\mathbb Z}
\newcommand{\R}{\mathbb R}
\newcommand{\quaternions}{\mathbb H}
\newcommand{\M}[1]{M_{#1}(\C)} 
\newcommand{\Tr}{\mbox{\rm Tr}}
\newcommand{\BH}{\ensuremath{B(\H)}} 
\newcommand{\tensor}{\otimes} 
\newcommand{\BHH}{\ensuremath{B(\H\tensor\H)}} 
\newcommand{\entrelacement}{\/\textit{entrelacement\xspace}} 
\newcommand{\isom}{\cong}
\newcommand{\Sh}{\ensuremath{\widehat{S}} } 
\newcommand{\Ah}{\ensuremath{\widehat{A}} } 
\newcommand{\Bh}{\ensuremath{\widehat{B}} } 
\newcommand{\AhA}{\ensuremath{\Ah\tensor A}}
\newcommand{\BhB}{\ensuremath{\Bh\tensor B}}
\newcommand{\uh}{\ensuremath{\widehat{u}}} \newcommand{\Uh}{\ensuremath{\widehat{U}}} 
\newcommand{\fancyK}{\ensuremath{\widetilde{K}}} 
\newcommand{\antipode}{\ensuremath{\kappa}} 
\newcommand{\antipodeh}{\ensuremath{\widehat\kappa}} 
\newcommand{\ksymmetric}{$\antipode$-symmetric\xspace}
\newcommand{\kpositive}{$\antipode$-positive\xspace}
\newcommand{\kpositivity}{$\antipode$-positivity\xspace}
\newcommand{\counit}{\epsilon} 
\newcommand{\counith}{\widehat{\epsilon}}
\newcommand{\F}{\mbox{\ensuremath{\mathcal F}}} 
\newcommand{\Fh}{\mbox{\ensuremath{\widehat{\mathcal F}}}} 
\newcommand{\inv}{\ensuremath{{}^{-1}}}
\newcommand{\Id}{\mbox{\rm Id}}
\newcommand{\Ad}{\mbox{\rm Ad}}  
\newcommand{\Fix}{\mbox{\rm Fix}}  
\newcommand{\CoZ}{\mbox{\rm CoZ}} 
\newcommand{\CoZr}{\CoZ_{R}} 
\newcommand{\BiU}{\mbox{\rm Bi\hspace{1pt}\ensuremath{\mathcal{U}}}} 
\newcommand{\cofixe}{\/\textit{cofix\'e }} 
\newcommand{\fixe}{\/\textit{fix\'{e} }}
\newcommand{\Inn}{\mbox{\rm Inn$_H$}}  
\newcommand{\Ug}{{\mathcal U}}
\newcommand{\arrow}{\longrightarrow}
\renewcommand{\dim}{{\mbox{\rm Dim}_\C}}
\newcommand{\ps}{pseudosimple\xspace}
\newcommand{\pscity}{pseudosimplicity\xspace}
\newcommand{\commutes}{\copyright} 
\newcommand{\product}[1][\empty]{\ensuremath{\pmb{\boldsymbol{\cdot\!}}_#1}}
\newcommand{\coproduct}[1][\empty]{\ensuremath{\delta_#1}}
\newcommand{\norm}[1]{\left\|#1\right\|}
\newcommand{\lfnorm}[1]{\norm{#1}_{\scriptscriptstyle{lf}}}
\newcommand{\tracenorm}[1]{\norm{#1}_{\scriptscriptstyle{1}}}
\newcommand{\Ksymbol}{\ensuremath{\fancyK\mbox{-symbol}}\xspace}
\newcommand{\Ksymbols}{{\Ksymbol}\mbox{s}\xspace}
\newcommand{\convolution}{\diamond} \newcommand{\convolute}{\convolution}  \newcommand{\convolve}{\convolution} 
\newcommand{\boxproduct}{\,\square\,} 
\newcommand{\comment}[1]{\-\marginpar[\raggedleft\footnotesize\it\textcolor{Sienna4}{#1\smallskip}]{\raggedright\footnotesize\it\textcolor{Sienna4}{#1\smallskip}}}
\renewcommand{\comment}[1]{}
\newcommand{\compact}{\ensuremath{\mathcal K}}
\newcommand{\ip}[2]{\ensuremath{\left\langle #1,#2\right\rangle}}
\newcommand{\Cu}{\ensuremath{{\mathcal C}\hspace{-.75pt}u}}
\renewcommand{\fancyK}{\ensuremath{{K}}} 
\newcommand{\ConvolutionAlgebra}{\ensuremath{{\mathscr C}}}
\newcommand{\ConvolutionAlgebraProduct}{\ensuremath{{\ast}}}
\renewcommand{\S}{\ensuremath{\mathscr S}} 
\definecolor{refkey}{cmyk}{0.1,0.1,1,0}
\definecolor{labelkey}{cmyk}{0,0.2,1,0.1}
\hyphenation{mult-i-pli-cat-ive  co-mult-i-pli-cat-ive anti-mult-ipli-ca-tive co-anti-mult-ipli-ca-tive iso-mor-phism co-anti-iso-mor-phism anti-auto-mor-phism anti-auto-mor-phism auto-mor-phi-sms anti-auto-mor-phism }
\section{Introduction}
\marginpar{ 
	}

Classification is a recurring theme in mathematics, and probably the most successful approach to classifying  C*-algebras has been to use $K$-theory, often  augumented by some additional information, as a classifying functor\cite{classification.overview}. 
We consider the case of C*-algebras with Hopf algebra structure, and 
we find that in many cases there is a product structure on the $K$-theory group. We then address the problem of constructing Hopf algebra maps from algebra maps respecting the product structure on the $K$-theory group, with applications to constructing automorphisms and isomorphisms of Hopf algebras. Theorems \ref{th:banacheweski.Hopf},   \ref{th:banacheweski.fd.Hopf2} and \ref{th:banacheweski.Hopf2} allow constructing Hopf algebra (co-anti) automorphisms or isomorphisms from purely K-theoretical data. These results are used to study bi-inner Hopf *-automorphisms, which are the Hopf *-automorphisms of a Hopf C*-algebra that are inner as algebra automorphisms, both in the dual algebra and in the given algebra. We extend the results of \cite{kucerovsky.BiU}.  We also develop techniques of independent interest, involving Jordan bi-algebra maps, linear positivity preserving maps, and other related concepts. In the last two sections, we give an example where we apply the techniques to the case of a Hopf AF C*-algebra, which includes an interesting special case that we call the quantum UHF case. With a compact Hopf AF C*-algebra can be associated two K-theory rings, one on the discrete dual and one on the algebra itself. We obtain isomorphism results for both cases (Theorem \ref{th:iso.for.discrete.AF} and Theorem \ref{th:iso.for.compact.AF}, respectively).
 
A Hopf algebra is a bi-algebra with an antipode map $\antipode.$ A multiplier Hopf algebra is a generalization where the co-product homomorphism takes values in a multiplier algebra.
Consider a compact  Hopf algebra $A$ that is also a C*-algebra; in the framework of \cite{BS}, the dual object is a C*-algebra $B,$ and has a co-product $\coproduct{}\colon B\arrow \Mult(B\tensor B).$  This dual object will be  a possibly infinite direct sum of matrix blocks at the level of C*-algebras, and the co-product homomorphism takes values in a direct product of matrix blocks. 
Our notation is based on that of \cite{BS}, denoting co-products by $\coproduct{},$ antipode by $\antipode,$ pairing by $\beta(\cdot\,,\cdot),$ and  co-unit by $\counit.$ The real algebra of \ksymmetric elements is the algebra of elements that are fixed under the antipode composed with the C*-algebraic involution.  The co-centre is the sub-algebra of co-commutative elements. We denote the flip, on a tensor product, by $\sigma.$

\section{K-theory of discrete multiplier Hopf C*-algebras}

We begin with the module picture of (algebraic) $K$-theory.

\begin{proposition} Let $M_1$ and $M_2$ be two projective sub-modules of a discrete Hopf C*-algebra $A.$ Taking the tensor product of these modules and using the co-product homomorphism to  restrict rings gives a projective module over $A.$   \end{proposition}
\begin{proof}
The module $M_i$ is a sub-module of a $c_0$-direct sum of matrix algebras. Such modules are necessarily finite-dimensional (over $\C$), and so the tensor product $M_1\tensor M_2$ of two such projective sub-modules is again finite-dimensional. 
The tensor product $M_1\tensor M_2$ is in a natural way a module over the multiplier algebra $\Mult(A\tensor A),$ and thus we may restrict by the co-product homomorphism. The restriction of rings operation will not increase the dimension, and therefore we obtain a sub-module of $A$ with finite complex dimension.  Such a sub-module is supported on finitely many matrix blocks of $A,$ and thus is projective.
\end{proof}
Since the restriction of rings operation is functorial, we obtain not just a product operation on the projective sub-modules of $A,$ but also a  product operation on the $K_0$-group in algebraic $K$-theory. In the case of $c_0$-direct sums of matrix algebras, the algebraic $K_0$-group is isomorphic to the Banach algebra $K_0$-group that is usually used with C*-algebras. 
 Therefore, we may further assume that we are working with Hilbert modules rather than just algebraic modules. 

There exists also an operator picture of the $K$-theory group, in terms of projections in $M_n(A).$ 
In $K$-theory it is necessary that direct sums of generators may be formed, and this is done by taking direct sums within a matrix algebra.  In order to accomodate different sizes of matrix algebras in a unified way, we may as well regard $M_n$ as a subalgebra of the compact operators $\compact.$ 
Viewing the product in terms of operators in $A\tensor\compact$ rather than modules, it satisfies the following properties:
\begin{definition} We denote by $\boxproduct$ a binary operation on  elements of $ A\tensor \compact$ with the following properties:
\begin{enumerate}\item (Pullback) $a(b\boxproduct c)=\sum (ba_1)\boxproduct(ca_2),$ where $\coproduct{}(a)=\sum a_1\tensor a_2,$ and
	\item (normalization) $(\tau\tensor t)(b\boxproduct c)=(\tau\tensor t)(b)(\tau\tensor t)(c),$ where $\tau$ denotes the  Haar state of $A$ and $t$ denotes the standard trace on $\compact$.
\end{enumerate} 
\end{definition}
 We remark that the above pullback property comes from the fact that, when we perform the restriction of rings operation, the action of $a$ in $A$ on the restricted module is mapped to the action of $\coproduct{}(a)$ on a tensor product of $A$-modules.  

We recall that the dual object of a discrete multiplier Hopf C*-algebra is, in the setting provided by \cite{BS,BBS}, a compact Hopf C*-algebra. There is a Fourier transform, see \cite[pg. 394]{PW1990} and \cite[para. 1.3]{BBS}, that can be defined as $\F(b):=[\Id\tensor\tau(b\cdot)] V^* ,$ where $V$ is a multiplicative unitary, and $\tau$ is the (extension of the) Haar weight.  The Fourier transform can also be defined via the pairing with the dual (following  Van Daele\cite{vanDaele1994}) by  
$$\beta(a,\F(b))=\tau(ba),$$ where  the elements $a$ and $b$ belong to a Hopf C*-algebra $A,$ $\tau$ is the Haar weight,  and $\beta(\cdot\,,\cdot)$ is the pairing with the dual algebra. In most cases, we will assume that the  Haar weights are tracial (we will see that this allows us to bring real C*-algebraic $K$-theory into the picture). 
An operator-valued convolution product, $\convolve,$ can be defined by the  property $\F(a\convolve b)=\F(a)\F(b),$ where $a$ and $b$ are elements of a Hopf C*-algebra $A,$ and $\F$ is the Fourier transform defined previously. 

 The next Proposition relates  $\boxproduct$ and $\convolve.$  

\begin{proposition} Let $A$ be a discrete Hopf C*-algebra with tracial Haar weight. Let $\ell_1$ and $\ell_2$ be in $A^+.$
If $\sigma$ is a trace in $T(A),$  and $t$ is the standard trace on $\compact,$ then
 $(\sigma\tensor t)(\ell_1\boxproduct \ell_2)=\sigma(\ell_1 \convolve \ell_2).$  \label{prop:box.and.convolve}
\end{proposition}
\begin{proof} Let us denote by $\tau$ the tracial Haar state on $A.$ Note that $\tau(\ell_1 \convolve \ell_2)=\tau(\ell_1)\tau(\ell_2),$ and also $(\tau\tensor t) (\ell_1\boxproduct \ell_2)=\tau(\ell_1)\tau(\ell_2),$ so that $(\tau\tensor t)(\ell_1\boxproduct \ell_2)=\tau(\ell_1 \convolve \ell_2).$ Thus the desired identity holds in the case $\sigma=\tau.$  

 Let us next consider the case where $\sigma$ is of the form $\tau(z\cdot)$ with $z$ being central and supported in only finitely many matrix blocks. If $\coproduct{}(z)=a_{(1)}\tensor a_{(2)},$
we have 
\begin{align*}(\sigma\tensor t) (\ell_1\boxproduct \ell_2)&=(\tau\tensor t)((a_{(1)}\ell_1)\boxproduct (a_{(2)}\ell_2))\\
&=\tau (a_{(1)}\ell_1) \tau(a_{(2)}\ell_2)\\
&=\tau((a_{(1)}\ell_1)\tensor (a_{(2)}\ell_2)).\end{align*}
It follows that   $(\sigma\tensor t)(\ell_1\boxproduct \ell_2)$ is equal to $(\tau\tensor\tau)(\coproduct{}(z)(\ell_1\tensor\ell_2)).$ Using the definition of the Fourier transform to rewrite this in terms of the pairing $\beta,$ we obtain $\beta(\coproduct{}(z),\F(\ell_1)\tensor\F(\ell_2)),$ which is equal to $\beta(z,\F(\ell_1)\F(\ell_2)).$ 

 From the definition of the Fourier transform, $$\beta(z,\F(\ell_1)\F(\ell_2))=\tau(z\F\inv(\F(\ell_1)\F(\ell_2))).$$ But $\F\inv(\F(\ell_1)\F(\ell_2))$ is the convolution product of $\ell_1$ and $\ell_2,$ so we have proven the desired identity in the case  where $\sigma$ is of the form $\tau(z\cdot)$ with $z$ being central and supported in only finitely many matrix blocks. 

In the general case where $\sigma$ is an unbounded trace on $A,$ we use the fact that traces on a $c_0$ direct sum of matrix algebras are lower semicontinuous to write the given $\sigma$ as a limit of bounded and compactly supported traces $\sigma_n.$  For these traces, $(\sigma_n\tensor t) (\ell_1\boxproduct \ell_2)=\sigma_n(\ell_1 \convolve \ell_2),$ from which the general case follows. 
\end{proof}

It is natural to consider maps from one discrete Hopf C*-algebra to another that intertwine the product $\boxproduct.$ We state the definition in a dualized form, using the fact that the states on the $K$-theory group $K(A)$ correspond to the traces on the algebra $A$. In fact we only need this property to hold in the case that $p$ and $q$ are generators of $\fancyK$-theory, which in our case are minimal projections of the algebra.
\begin{definition}A map $f\colon \fancyK(A)\arrow \fancyK(B)$ is said to be \emph{$K$-co-multipli\-cative} if  $  (\sigma\tensor\tau)(f(p)\boxproduct f(q))=(\sigma\tensor\tau)((f\tensor\Id)(p\boxproduct q)),$ where $p$ and $q$ are 
generators of $\fancyK$-theory,    the linear functional $\sigma$ is a trace on $B,$ and the linear functional $\tau$ is the standard trace on $\compact.$ \end{definition} 
We shall show that  $\fancyK$-co-multipli\-ca\-tive maps induce maps on the dual that are well-behaved with respect to co-traces. \comment{An unbounded example of a co-tracial linear functional is the trace on the dual.
I think that this trace is in some cases the co-unit composed with the inverse Fourier transform, but since we may not have co-unit projs we have to take limits of co-tracial functionals of the form $\tau(e_n\cdot)$ where $(e_n)$ is an approximate unit for the co-commutative subalgebra.}
We say that a \textit{co-tracial linear functional} is a linear functional $g$ with the property $g(a\convolve b)=g(b\convolve a).$ Since the Haar state satisfies $\tau(a\convolve b)=\tau(a)\tau(b),$ it is an example of a co-tracial linear functional.

\begin{lemma} Let $A, B$ be Hopf C*-algebras with  tracial Haar states. Let $f:A\arrow B$ be a C*-algebraic  isomorphism that intertwines the Haar states. Then, we have the identity
  $$\F_B \compose f = f^*\inv\compose\F_A,$$ where $\F$ denotes the Fourier transform, and $f^*\inv,$ which is continuous, is the inverse map for $f^*,$ the map induced by $f$ on the dual algebras.  
\label{lem:pullbacks.and.Fourier}\end{lemma}
 
\begin{proof}
The Fourier transform on $A$  can be written as $\F_A\colon a \mapsto \tau_A (a\cdot).$  A similar statement holds for the Fourier transform on $B.$ Then, using the property that $\tau_B(f(x))=\tau_A(x)$ for all $x\in A$, we get
\begin{align*}
f^*\compose (\tau_B (f(a) \cdot)) &=  \tau_B(f(a) f(\cdot))\\
&= \tau_B(f(a (\cdot)))\\
&= \tau_A(a \cdot).\end{align*}
From the above it follows that $f^*\compose\F_B \compose f = \F_A.$ 
That $f^*\inv$ exists follows from the fact that $f$ is invertible and the inverse induces a map of dual algebras. That $f^*\inv$ is continuous is seen by applying the open mapping theorem to $f^*.$ Composing  $f^*\compose\F_B \compose f = \F_A$ with   $f^*\inv$ from the left, we have the desired conclusion. 
\end{proof}

\begin{lemma} Let $A$ and $B$ be discrete multiplier Hopf C*-algebras that have a tracial Haar weight. 
Let $f:A\longrightarrow B$ be a C*-isomorphism that intertwines Haar states, and whose induced map on $K$-theory is  $\fancyK$-co-multipli\-ca\-tive. \label{lem:cotraces}\label{lem:de-projectivization}
The map $f^*$ induced by $f$ on the dual algebra(s) satisfies 
$$g(f^*\inv(y_1 y_2))= g(f^*\inv(y_1) f^*\inv (y_2))$$
 for all $y_i$ in the dual algebra $\Ah$,  and all bounded cotracial linear functionals $g.$ 
\end{lemma}
\begin{proof} Let $T$ be a trace on $B,$ and let $\tau$ be the Haar state of $B.$ 
From the hypothesis that the given map $f$ respects the product $\boxproduct,$ we  have  $$ (T\tensor\tau)(f(p)\boxproduct f(q))=(T\tensor\tau)((f\tensor\Id)(p\boxproduct q)),$$
where $p$ and $q$  are minimal projections of $A$. By Proposition \ref{prop:box.and.convolve}, it follows that
$$ T(f(p)\convolve f(q))=T(f(p\convolve q)),$$ where $\convolve$ denotes  convolution of operators.

Passing to linear combinations of traces, we thus have, for all tracial linear functionals $T$ on $B$, that
$$ T(f(p)\convolve f(q))=T(f(p\convolve q)).$$ 
Since the tracial linear functional $T$ can be written in terms of the Haar state $\tau$ and a central element $z$ as $T(x)=\tau(zx),$ it follows from the definition of the Fourier transform, $\beta(a,\F(b))=\tau(ba),$ that $T(x)=\beta(z,\F(x)).$ Putting this into the above equation, we obtain
$$ \beta(z,\F(f(p))\F(f(q)))=\beta(z,\F(f(p\convolve q))).$$

We use Lemma \ref{lem:pullbacks.and.Fourier}  to express the action of $f$ on the dual algebras in terms of the Fourier transforms, obtaining  $\F\compose f= f^*\inv\compose\F.$ We conclude that, denoting the Fourier transforms of $p$ and $q$ by $\hat{p}$ and $\hat{q}$ respectively,
$$ \beta(z,f^*\inv ( \hat{p} ) f^*\inv (\hat{q}))=\beta(z,f^*\inv(\hat{p} \hat{q}))$$
for any central element $z\in A.$ 

Regarding the above as an equality of bilinear forms in $\hat{p}$ and $\hat{q},$ we notice that we can replace  $\hat{p}$ and $\hat{q}$ by  finite linear combinations of elements  having the same properties as, respectively, $\hat{p}$ and $\hat{q}.$ Since any element of $A$ that is supported in finitely many matrix blocks of $A$ can be written as a finite linear combination of minimal projections, we conclude that 
\begin{equation}g(f^*\inv(y_1 y_2))= g(f^*\inv ( y_1 ) f^*\inv (y_2)),\label{eq:g}\end{equation}
for any cotracial functional $g$ and all $y_i\in \Ah$ given by the Fourier transform of an element supported in finitely many matrix blocks.
 But then the equation $g(f^*\inv(y_1 y_2))= g(f^*\inv ( y_1 ) f^*\inv (y_2))$ holds for all $y_i$ in a dense subset of the unital C*-algebra $\Ah,$ and since $g$ and $f^*\inv$ are continuous, the equation then holds for all $y\in \Ah.$ 
 \end{proof}

We now have need of an improvement of a result from \cite{Wolff}. In the statement of our result, the notation $\prime$ is used to denote a relative commutant  within $B.$ 
\begin{proposition} Let $A$ and $B$ be C*-algebras, with $A$ unital. Let $f\colon A\arrow B$ be a linear and positive map that takes orthogonal positive operators to orthogonal positive operators. 
Let $C:=\overline{f(1_A)\{f(1_A)\}'}.$ Then, $f(A)\subseteq C$ and there is a Jordan homomorphism $\pi\colon A\arrow\Mult(C)$ such that
$f(a)=f(1_A )\pi(a)$ for all $a\in A.$\label{prop:Wolff.positive.case}\end{proposition}
\begin{proof}
If we can show that $f$ takes self-adjoint orthogonal elements to self-adjoint orthogonal elements, the conclusion we want is then
provided by \cite[Theorem 2.3]{Wolff}.

We may as well suppose that $\norm{f}\leq1.$ Let $a$ and $b$ be orthogonal self-adjoint elements of $A.$ The sub-C*-algebra generated by $1$, $a$, and $b$ is abelian and will be denoted $C(X).$ Denoting the restriction of $f$ to $C(X)$ by $\phi,$ we note that $\phi$ is completely positive (see \cite{Stinespring}).
By the Cauchy-Schwartz inequality for completely positive maps, $$\phi(a^2)\geq\phi(a)^2\geq0.$$ By hypothesis, $\phi(a^2)$ and $\phi(b^2)$ are orthogonal positive elements, so cutting down the above inequality on both sides by the positive element $\phi(b^2),$ we conclude that $\phi(a)$ is orthogonal to $\phi(b^2).$ Similarly, cutting down $\phi(b^2)\geq\phi(b)^2\geq0$ on left and right by $\phi(a),$ we then have as required that $\phi(a)$ and $\phi(b)$ are orthogonal self-adjoint elements.
\end{proof}
 \begin{proposition}Let $A$ and $B$ be discrete multiplier Hopf C*-algebras that have tracial Haar weights.  Let $f\colon A\longrightarrow B $ be a $\fancyK$-co-multiplicative C*-isomorphism that intertwines co-units and antipodes.
  Then
the pullback map $f^*$ is a Jordan homomorphism \label{prop:box.implies.jordan}\end{proposition}
 \begin{proof} The Haar states extend to the canonical trace on the GNS representation, each matrix block being represented with multiplicity given by its matrix dimension. Since the Haar states are thus determined by C*-algebraic data, the C*-isomorphism $f$ must  intertwine the canonical traces, and hence intertwines the Haar states. 
Applying Lemma \ref{lem:cotraces} to $f\inv$, we have that $f^*$ is at least multiplicative under cotracial linear functionals. In fact, replacing bi-linear forms by $n$-linear forms in the proof of Lemma \ref{lem:cotraces}, we have   
\begin{equation}g(f^*(b_1)f^*(b_2)f^*(b_3) \cdots f^* (b_n))=g(f^*(b_1b_2b_3\cdots b_n)) \label{eq:approx.mult}\end{equation}
for all cotracial functionals $g$ and all $b_i\in\Bh.$. We note that $f^*$ is a linear and unital map that maps  self-adjoint elements to self-adjoint elements,  and intertwines the canonical traces $\tau$ coming from the GNS representations. We deduce from the above equation that
$$\tau(f^*(x)^n)=\tau(f^*(x^n)),$$
where $x$ is in the compact Hopf C*-algebra $\Bh,$ and $\tau$ is the tracial (and co-tracial) Haar state of the compact Hopf C*-algebra $\Ah.$ Since $\tau$ is a bounded linear functional, it follows from the Stone-Weierstrass theorem on $[0,1]$ that $\tau(g(f^*(x)))=\tau(f^*(g(x)))$ for all self-adjoint elements $x$ and all continuous functions $g.$  From this we have ${\rm sp}(f^* (x))\subseteq {\rm sp}(x).$ The reason is that otherwise there would exist a continuous function $g$ such that $g(f^*(x))$ is positive and nonzero while $g(x)$ is zero, which would contradict the fact that $\tau$ is a faithful linear functional. 
Applying the same argument to the inverse map, we conclude that in fact $f^*$ preserves the spectrum of self-adjoint operators. Since a self-adjoint operator with spectrum contained in $[0,\infty)$ is positive, it follows that $f^*$ maps a positive operator to a positive operator.

Since by equation \eqref{eq:approx.mult} we have $\tau(f^*(b_1)f^*(b_2))=\tau(b_1b_2),$ it follows that if $b_1$ and $b_2$ are orthogonal positive operators, then $f^*(b_1)$ and $f^*(b_2)$ are positive operators which are tracially orthogonal. But, denoting these operators by $a_1$ and $a_2$, we then have 
$$0=\tau(a_1a_2)=\tau((a_2)^{1/2} a_1 (a_2)^{1/2}),$$ from which it follows, by the faithfulness of the linear functional $\tau,$ that $a_1^{1/2}a_2^{1/2}=0,$
implying that $a_1$ and $a_2$ are orthogonal positive operators. 
Thus, $f^*$ maps orthogonal positive operators to orthogonal positive operators, and by Proposition \ref{prop:Wolff.positive.case} is a Jordan homomorphism.
\end{proof}

We now give a slight generalization of  \cite[Lemma 2.7]{kucerovsky.BiU}. We note that in the following the pullback is required to be a Jordan *-homomorphism, not just a Jordan homomorphism.
\begin{lemma} Let $A$ and $B$ be  discrete  Hopf C*-algebras with tracial Haar weights. Let $\alpha\colon A\arrow B$ be a  *-isomorphism, and let $\hat{\alpha}\colon \Bh\arrow\Ah$ be its action on the dual. We suppose the action $\hat{\alpha}$ on the dual is a Jordan *-isomorphism. Then either $\hat\alpha$ is multiplicative, or $\hat\alpha$ is anti-multiplicative. \label{lem:jordan.auto} The same conclusion holds if discrete is replaced by compact.\end{lemma}
\begin{proof} A Jordan *-isomorphism maps the C*-norm  unit ball onto itself. If $\hat{\alpha}$ maps the unit ball onto itself, this implies that the map that it induces on linear functionals is an isometry with respect to the usual dual norm (on linear functionals.) Thus, the map $\alpha$ is an isometry with respect to this norm, which makes $\alpha$ an isometry of the pre-dual in the sense of \cite[Th\'eor\`eme 2.9 ]{CES}. It follows that $\hat\alpha$ is either multiplicative or anti-multiplicative, as claimed.
\end{proof}

\begin{theorem} Let $A$ and $B$ be discrete Hopf C*-algebras with  tracial Haar weights. Let $f\colon A\longrightarrow B$ be
a  C*-isomorphism that intertwines antipodes and co-units. We suppose that the induced map on $K$-theory intertwines the  products $\boxproduct_A$ and $\boxproduct_B$. Then $A$ and $B$ are isomorphic or co-anti-isomorphic as Hopf algebras.\label{th:banacheweski.Hopf}
\end{theorem}
 \begin{proof} By Proposition \ref{prop:box.implies.jordan} we have that under the hypothesis the pullback $f^*\colon\Bh\longrightarrow\Ah$ of $f$ is  a Jordan homomorphism at the level of C*-algebras. The fact that $f$ intertwines antipodes insures that the pullback is a Jordan *-homomorphism (more precisely, a Jordan *-isomorphism). By Lemma \ref{lem:jordan.auto} the pullback map $f^*$ is either multiplicative or anti-multiplicative. We thus have, by duality, that $f$ is either an isomorphism or a co-anti-isomorphism of bi-algebras.
 It follows from uniqueness of the Hopf algebra antipode(s) that $f$ is a Hopf algebra (co-anti)isomorphism. 
 \end{proof}

The above theorem can be re-stated in terms of maps on $K$-theory groups. If we also assume finite-dimensionality,  then this takes on an especially attractive form. Without such additional assumptions, we must for technical reasons introduce some real $K$-theory. 

From the classification theory of AF-algebras\cite{elliott}, we have  a well-known one-to-one correspondence between *-homomorphisms of AF C*\-algebras and homomorphisms of ordered $\fancyK$-theory groups (at least up to certain natural notions of equivalence). 
Each involutory *-anti\-automor\-phism of an AF algebra $A$ has associated with it a real subalgebra $R_A$ of the AF algebra. These subalgebras are called real AF algebras and there exists also a classification of real AF algebras\cite{giordano}, and thus, for a fixed pair of AF algebras with given involutions, there
exists a natural correspondence between the following three objects:
\begin{enumerate}
	\item homomorphisms of ordered real $K$-theory groups, $f\colon K_R (A)\arrow K_R (B),$ 
	\item *-homomorphisms of real subalgebras, $\phi\colon R_A\arrow R_B,$ and
	\item *-homomorphisms $\phi\colon A\arrow B$ that intertwine the given involution(s) of $A$ and $B.$
\end{enumerate}
As before, certain natural notions of equivalence are implicit in the above. Morover, the real $K$-theory of a (real) AF algebra actually consists of three groups, rather than just one as in the complex case. Taking now the case of a discrete multiplier Hopf C*-algebra, the real subalgebra coming from an involutive antipode is the subalgebra of \ksymmetric elements, and the product $\boxproduct$ gives a product on the complexification of $K_{0,R} (R_A).$ (It does not appear to be true in general that the product of real elements is real.) We now have sufficient terminology to state our next theorem:

\begin{theorem} Let $A$ and $B$ be discrete multiplier Hopf C*-algebras with  tracial Haar weights. Suppose $f\colon K_R (R_A)\arrow K_R (R_B)$ is an isomorphism of real $K$-groups whose complexification intertwines the products on  $K_0 (A)$ and $K_0 (B).$ 
We suppose also that $f$ intertwines the $K$-theory states induced by the co-units.  Then $A$ and $B$ are isomorphic or co-anti-isomorphic as Hopf algebras.\label{th:banacheweski.Hopf2}
\end{theorem}
 \begin{proof} The given map $f$ induces an isomorphism of complex $K$-groups $f\colon K_0 (A)\arrow K_0 (B).$  Since the given map respects the real substructures associated with\footnote{The antipodes are involutive if the Haar weights are tracial\protect{\cite[Prop. 4.8b]{BS}}} the antipodes, we obtain from the classification theory of real C*-algebras an algebra *-isomorphism $\phi\colon A\arrow B$ that intertwines antipodes. Moreover, it follows by construction that the map induced by $\phi$ on $K$-theory intertwines the products $\boxproduct_A$ and $\boxproduct_B.$ By hypothesis, $\phi$ maps the support projection for the co-unit to the support projection for the co-unit, so that $\phi$ intertwines co-units, and we may apply Theorem   \ref{th:banacheweski.Hopf} to conclude that $A$ and $B$ are either isomorphic or co-anti-isomorphic as Hopf algebras.
 \end{proof}
 
\section{When do Jordan maps of Hopf algebras intertwine antipodes?}
In this section, we give an interesting  result that allows us to further simplify the results of the previous section, and that is moreover of independent interest. 

A Jordan homomorphism is a not necessarily multiplicative linear map that has the property $\phi(a^2)=\phi(a)^2$ for all elements $a.$
It seems interesting to ask the following question: if an algebra map of  Hopf C*-algebras induces a Jordan map on the dual algebras, does it intertwine antipodes? This is motivated by the known fact that a bi-algebra map of Hopf algebras automatically intertwines antipodes. We give an affirmative answer in a special case.

We recall that in the setting of C*-algebras, it has been shown that a Jordan homomorphism is always a  direct sum of automorphisms and antiautomorphisms. (For matrix algebras, this was shown by Jacobson and Rickart\cite{JR2,JR1}. Kadison\cite{kadison.isometries,kadison2} extended these results for surjective maps onto C*-algebras and von Neumann algebras, and showed also that bijective order isomorphisms are Jordan isomorphisms. St\o rmer showed\cite{stormer1965} that Jordan homomorphisms of C*-algebras into C*-algebras are sums of homomorphisms and anti-homomorphisms.)

We recall, following \cite[pg.61]{abe}, that if $C$ is a co-algebra over $\C$ and $A$ is a unital algebra over $\C,$ the space of complex--linear functions from $C$ to $A$ can be made into a convolution algebra, denoted $\ConvolutionAlgebra(C,A).$ The product is defined by $(f\ConvolutionAlgebraProduct g)(c):= \sum f(c_1)g(c_2)$ where $\coproduct{} (c)=\sum c_1 \tensor c_2.$ There exists an identity element, namely $1_A\counit_C.$

The proof of the next Proposition is based on a technique from \cite[pg.152]{DNR}.
\begin{proposition} Let $A$ and $B$ be compact Hopf C*-algebras, with  tracial Haar states. Let $f\colon A\longrightarrow B$ be an algebra map, intertwining co-units. Let the induced map $f^*\colon \Bh\longrightarrow\Ah$ on the duals be Jordan. It then follows that $f$ intertwines antipodes.\label{prop:Hopf.Jordan.maps.preserve.antipodes}
\end{proposition}
\begin{proof}
The induced map on the duals, $f^*\colon \Bh\arrow\Ah,$ where $\Bh$ and $\Ah$ are $c_0$-direct sums of matrix algebras, is Jordan and thus has the property that there exists a multiplier projection $P\in\Mult(\Ah)$ such that $Pf^*\colon \Bh\arrow\Ah$ is multiplicative, and $(1-P)f^*\colon \Bh\arrow\Ah$ is anti-multiplicative. This is shown in \cite{stormer1965}, where in our case the special structure of the algebras gives multiplier projections rather than central projections in the enveloping Kac-von Neumann algebra.

The projection $P$ obtained above belongs to the  centroid of $\Ah,$ and thus $P\Ah$ and $(1-P)\Ah$ form a pair of complementary ideals. In particular, they are annihilators of ideals of $\Ah$, and by \cite[Theorem 2.3.1]{abe} 
they correspond under duality to sub-co-algebras of $A,$ denoted $\S_1 (A)$ and $\S_2 (A)$ respectively. Restricting the algebra map $f\colon A \longrightarrow B$ to $f\colon \S_1(A) \longrightarrow B$ we thus obtain a co-multiplicative map, and restricting to $f\colon \S_2(A) \longrightarrow B$ we obtain an anti-co-multiplicative map.  

Now consider the convolution algebra $\ConvolutionAlgebra(A,B).$ We shall show that $\antipode_B\compose f$ and $f\compose\antipode_A$ are equal as elements of this algebra, in which case it follows that $f$ intertwines antipodes as claimed.  Taking $w\in \S_1(A),$ and writing $\coproduct{}(w) = \sum w_1\tensor w_2,$ the co-multiplicativity of $f$ on $w$ gives
\begin{align*}
\left[(\antipode_B\compose f)\ConvolutionAlgebraProduct f\right] &= \sum \antipode_B (f(w_1))f(w_2)\\
							&= \counit_B (f(w))1_B\\
							&= \counit(w)1_B.\\\end{align*}
We recall that if the Haar state is tracial, the antipode is involutive \cite[Prop. 4.8b]{BS}, see also \cite{woronowicz}. Then the identity 
$\coproduct{}\compose\antipode=\sigma\compose(\antipode\tensor\antipode)\compose\coproduct{},$
where $\sigma$ is the flip, together with 
the fact that $f$ is anti-co-multiplicative on $\S_2(A)$ gives
\begin{align*} (f\tensor f)\coproduct{A} (w) &=\sigma\compose\coproduct{B} (f(w))\\
              &=(\antipode_B\tensor\antipode_B)\coproduct{B}(\antipode_B (f(w)),\end{align*}
							for all $w\in \S_2(A).$
But, then, writing $\coproduct{} (w) = \sum w_1\tensor w_2,$ we have
\begin{align*}
\left[(\antipode_B\compose f)\ConvolutionAlgebraProduct f\right] &= \sum \antipode_B (f(w_1))f(w_2)\\
							&= \counit_B (\antipode_B (f(w)))1_B\\
							&= \counit_B (f(w))1_B\\
							&= \counit_B (w)1_B.\\\end{align*}
Since $\S_1(A)$ and $\S_2(A)$ are linear subspaces of $A$ whose span is all of $A$, we have by linearity that $\left[(\antipode_B\compose f)\ConvolutionAlgebraProduct f\right](w)=\counit_B (w)$ for all $w$ in $A.$ 

We next consider $\left[f\ConvolutionAlgebraProduct (f\compose\antipode_A)\right](w).$
In this case, for all $w\in A$, writing $\coproduct{}(w)=\sum w_1 \tensor w_2,$ we have
\begin{align*}
\left[f\ConvolutionAlgebraProduct(f\compose\antipode_A)\right](w)&=\sum f(w_1)f(\antipode_A(w_2))\\
							&= f\left(\sum w_1\antipode_A(w_2)\right)\\
							&= f(\counit_B (w)1_A\\
							&= \counit_B (w)1_B,\\\end{align*}
where we have only used the fact that $f$ is unital, co-unit preserving, and multiplicative.
We thus conclude that in the convolution algebra $\ConvolutionAlgebra(A,B),$ we have
\begin{align*}f\ConvolutionAlgebraProduct (f\compose\antipode_A)=1_\ConvolutionAlgebra
\intertext{and}
(\antipode_B\compose f)\ConvolutionAlgebraProduct f =1_\ConvolutionAlgebra
\end{align*}
from which it follows by the associativity of the operation $\ConvolutionAlgebraProduct$ that $\antipode_B\compose f = f\compose\antipode_A,$ as claimed. 
\end{proof}

We now improve a result from \cite{kucerovsky.FD.hopf}.
\begin{theorem} Let $A$ and $B$ be finite-dimensional Hopf C*-algebras. Let $f\colon A\longrightarrow B$ be
a  C*-isomorphism that intertwines co-units. We suppose that the induced map on $K$-theory  intertwines the  products $\boxproduct_A$ and $\boxproduct_B$. Then $A$ and $B$ are isomorphic or co-anti-isomorphic as Hopf algebras.\label{th:banacheweski.fd.Hopf2}
\end{theorem}
 \begin{proof} Finite-dimensional Hopf C*-algebras are Kac algebras and thus have tracial Haar states. The Haar states extend to the canonical trace on the GNS representation, each matrix block being represented with multiplicity given by its matrix dimension. Since the Haar states are thus determined by C*-algebraic data, the C*-isomorphism $f$ must  intertwine the canonical traces, and hence intertwines the Haar states. 
Applying Lemma \ref{lem:cotraces} to $f\inv$, we have   
\begin{equation*}g(f^*(b_1)f^*(b_2)f^*(b_3) \cdots f^* (b_n))=g(f^*(b_1b_2b_3\cdots b_n))\label{eq:approx.mult2}\end{equation*}
for all co-tracial functionals $g$ and elements $b_i\in\Bh.$ We deduce from the above equation that
$$\tau(f^*(x)^n)=\tau(f^*(x^n)),$$
where $x$ is in the compact Hopf C*-algebra $\Bh,$ and $\tau$ is the tracial (and co-tracial) Haar state of the compact Hopf C*-algebra $\Ah.$ The map $f^*$ is unital and  linear. Since  $x\in\Bh$ is  a linear transformation on the finite-dimensional Hilbert space associated with the Haar state, it has a set of eigenvalues, $\Lambda.$  Similarly, denote the eigenvalues of $f^*(x)$ by $\Lambda'.$ 
Since $\tau(f^*(x)^n)=\tau(f^*(x^n))=\tau(x^n)$ for all $n,$ it follows that as a set, $\Lambda\isom\Lambda'.$ This is because  the sums of powers of the eigenvalues are equal, and so the classical Newton-Girard identities imply that the eigenvalues are equal up to  permutation.
 
Applying the same argument to the inverse map, we conclude that in fact $f^*$ preserves the spectrum of operators. A bijective spectrum-preserving map of finite-dimensional C*-algebras is necessarily a Jordan map\cite[Th\'eor\`eme 1.12]{aupetit2}. Applying Proposition 
\ref{prop:Hopf.Jordan.maps.preserve.antipodes}, we conclude that the algebra map $f\colon A\arrow B$ intertwines antipodes. It then follows that $f^*\colon \Bh\arrow\Ah$ intertwines the C*-involutions, ${}^*\colon\Ah\arrow\Ah$ and ${}^*\colon\Bh\arrow\Bh.$ Since $f^*\colon \Bh\arrow\Ah$ is therefore a Jordan *-isomorphism (and not just a Jordan isomorphism), it follows that  we can apply Lemma \ref{lem:jordan.auto} to conclude that $f^*$ is either multiplicative or anti-multiplicative, from which the conclusion follows.
\end{proof}

\section{On bi-inner automorphisms}

Bi-inner Hopf algebra *-automorphisms are the Hopf *-automor\-phisms of a Hopf C*-algebra that are inner as algebra automorphisms, both in the dual algebra and in the given algebra. In this section, we will improve the following result:
\begin{theorem}[\protect{\cite{kucerovsky.BiU}}] The bi-inner Hopf algebra *-automorphisms of a finite-dimensional Hopf C*-algebra $A$ are precisely the component of the identity of the set of maps of the form
$x\mapsto vxv^*$ where $u$ is a {\ksymmetric}al unitary that commutes with the co-commutative sub-algebra of $A.$\label{th:biU.fd1}\end{theorem}

Proposition \ref{prop:Hopf.Jordan.maps.preserve.antipodes} of the previous section will in effect allow us to remove the condition in the above that the unitaries be in the algebra of {\ksymmetric}al elements (at least in the  finite-dimensional case). 

\begin{theorem}Consider a finite-dimensional Hopf C*-algebra $A.$ The following are equivalent:
\begin{enumerate}\item  The map $\alpha\colon A\arrow A$ is a bi-inner Hopf *-automorphism,
\item The map $\alpha$ can be written as $x\mapsto U^* x U,$ where $U$ is 
a unitary in the unitization of $A$ that commutes with co-central elements of $A.$ 
\end{enumerate}\end{theorem}
\begin{proof} If we are given a bi-inner Hopf *-automorphism, being inner, it must act on $A$ by some unitary $U$ in  $A,$ and it must moreover fix elements of the co-centre. From this it follows that it can be written $x\mapsto U^* x U,$ where $U$ is 
a unitary  that commutes with co-central elements of $A.$ 
Conversely, if we are given a unitary $U$ in  $A$ that commutes with co-central elements of $A,$  applying Theorem \ref{th:banacheweski.fd.Hopf2}  to the map $\alpha\colon x\mapsto U^* x U$ gives the required bi-inner automorphism. 
\end{proof}

We also point out that Theorem \ref{th:banacheweski.Hopf} has a corollary:

\begin{theorem} Consider a discrete Hopf C*-algebra $A$  with  tracial Haar state. Let $v\in A$ be a {\ksymmetric}al unitary that commutes with elements of the co-centre. Suppose also that $v$ is connected to the identity by a path of elements having the same properties. Then, the map $x\mapsto v^*xv$ is a Hopf *-automorphism.
\end{theorem}
\begin{proof} Theorem \ref{th:banacheweski.Hopf} shows that a map of the given form is either a Hopf *-automorphism or a Hopf *-co-anti-automorphism. However, the fact that the map is connected to the identity map rules out the case of a co-anti-automorphism.
\end{proof} 

\section{A detailed example: AF Hopf C*-algebras and quantum UHF algebras}
We
consider AF C*-algebras whose building blocks and connecting maps have Hopf structure. Thus the building blocks, being finite dimensional, are necessarily Kac algebras, and the C*-algebraic inductive limit is a C*-algebra with Hopf structure. The connecting maps need only be assumed to be bi-algebra maps, since bi-algebra maps of Hopf algebras are necessarily Hopf algebra maps \cite[pg.152]{DNR}.
We combine  our previous results with the classification theory for C*-algebras.

Proposition \ref{prop:Hopf.Jordan.maps.preserve.antipodes} gives a Corollary for AF Hopf C*-algebras:
\begin{corollary} Let $A$ and $B$ be  duals of compact AF Hopf C*-algebras. Let $f\colon A\longrightarrow B$ be an algebra isomorphism, intertwining co-units. Let the induced map $f^*\colon \Bh\longrightarrow\Ah$  be Jordan. It then follows that $f$ intertwines antipodes.\label{prop:Hopf.Jordan.maps.preserve.antipodes.AF}
\end{corollary}
\begin{proof}
Consider the induced map on the duals, restricted to a building block $D\subset \Bh.$  Then $f^*\colon D\arrow\Ah$ is an injective Jordan map. The restricted map $f^*\colon D\arrow\Ah$  is dual to a map $g\colon A\arrow \widehat{D}\subseteq B,$ where $\widehat{D}$ is  (by \cite[Theorem 2.3.1]{abe})  a finite-dimensional annihilator of an ideal, thus an ideal, in $B.$
Proposition \ref{prop:Hopf.Jordan.maps.preserve.antipodes} shows that the map $g\colon A\arrow\widehat{D}\subseteq B$ intertwines antipodes.   Choosing larger and large building blocks, we conclude that the given map $f\colon A\arrow B$ intertwines antipodes.
\end{proof}
This Corollary allows us to weaken the hypotheses of Theorem \ref{th:banacheweski.Hopf}, at least in the case of AF Hopf C*-algebras. We recall that at the C*-algebraic level, there exists a functorial classification of AF-algebras, given by an ordered $K$-theory group\cite{elliott}. By the term co-unit state, we denote the map on $K$-theory induced by the co-unit.
\begin{theorem} Let $A$ and $B$ be duals of compact AF Hopf C*-algebras. Let $f_{*}\colon K(A)\longrightarrow K(B)$ be
an isomorphism at the level of ordered $K$-theory that intertwines  the co-unit states and the  products $\boxproduct_A$ and $\boxproduct_B$. Then $A$ and $B$ are isomorphic or co-anti-isomorphic as Hopf algebras.\label{th:banacheweski.Hopf.AF.Hopf}\label{th:iso.for.discrete.AF}
\end{theorem}
\begin{proof} The above-mentioned classification theory of AF algebras gives an algebra *-isomorphism $f\colon A\arrow B.$ Proposition \ref{th:banacheweski.Hopf} shows that the dual $f^{*}$ of this map has the Jordan property, and Corollary \ref{prop:Hopf.Jordan.maps.preserve.antipodes.AF} shows that the map  $f\colon A\arrow B$ intertwines antipodes. Thus the dual map $f^{*}$ intertwines the C*-involutions. Lemma   \ref{lem:jordan.auto} then provides the required (co-anti)isomorphism.
\end{proof}
A concrete example of an interesting Hopf AF C*-algebra is given by choosing the building block $A_n$ to be the $n$-fold tensor product of the Kac-Palyutkin algebra\cite{KP1964} with itself. This algebra is an eight-dimensional  Kac algebra (or ring group) that is neither commutative nor co-commutative. Taking the C*-inductive limit we obtain what we term the quantum UHF C*-algebra associated with the Kac-Palyutkin algebra. A similar  construction can be made beginning with any finite-dimensional Kac algebra. 

Replacing the use of classification theory of AF C*-algebras in this section by some other $K$-theoretic classification of C*-algebras would lead to additional examples.

\section{K-theory and isomorphisms of compact AF Hopf C*-algebras}

Let $A$ be a compact AF Hopf C*-algebra. 
Recalling that the K-theory of a C*-algebraic  inductive limit is an algebraic direct limit of  K-theory groups, we see that if $P_1$ and $P_2$ are elements of $K(A)$ then they  belong to the $K$-theory $K(A_i)$ of some building block $A_i.$ Thus, the product  $P_1 \boxproduct P_2$ of these elements is in $K(A_i),$ and hence in $K(A).$ 

We note that a compact Hopf AF C*-algebra inherits an involutive antipode, and thus a tracial Haar state, from its finite-dimensional building blocks. 
 We say that an element of the dual algebra is compactly supported if it is zero outside all but finitely many matrix blocks of the dual. 

 \begin{proposition}Let $A$ and $B$ be compact Hopf AF C*-algebras.  Let $f\colon A\longrightarrow B $ be a $\fancyK$-co-multiplicative C*-isomorphism that intertwines co-units and Haar states.
  Then
the pullback map $f^*$ is a Jordan homomorphism \label{prop:box.implies.jordan.for.AF}\end{proposition}
  \begin{proof} Let $x$ be a compactly supported element of $\Ah.$ There then exists a building block $D\subseteq A$ such that $x\in\widehat{D},$ where
$\widehat{D}$ is a finite-dimensional annihilator of an ideal, thus an ideal, in $\Ah.$ 
	
The proof of  Lemma \ref{lem:cotraces} remains valid up to equation \eqref{eq:g}, so we see that the map $f^*$ induced by $f$ on the dual algebras satisfies 
$$g(f^*\inv(y_1 y_2))= g(f^*\inv(y_1) f^*\inv (y_2))$$
where the map $f^*\inv\colon\Ah\arrow\Bh$ is the inverse of the pullback of the given map, the elements $y_i$ are compactly supported elements in $\widehat{D}\subseteq\Ah$,  and $g$ is a  cotracial linear functional on $\Bh.$ The same holds for products of finitely many $y_i,$ so that
$$\tau_{\Bh} ({f\inv}^{*}(x)^{n})=\tau_{\Bh} ({f\inv}^{*}(x^n))=\tau_{\Ah} (x^n),$$
and $\tau_{\Ah}$ (resp. $\tau_{\Bh}$) is the  Haar state of the discrete Hopf C*-algebra $\Ah$ (resp. $\Bh$.) As in the proof of Theorem \ref{th:banacheweski.fd.Hopf2}, we conclude that ${f\inv}^*(x)$ and $x$ are isospectral. 

Approximating a general element of the $c_0$-direct sum of matrix algebras $\Ah$ by compactly supported elements, we see that the map $f^*\inv\colon\Ah\arrow\Bh$ preserves the spectrum of operators. 
A bijective spectrum-preserving map of  C*-algebras that are $c_0$-direct sums of matrix algebras is necessarily a Jordan map, as follows from \cite[Theorem 3.7]{aupetit3}.
	\end{proof}
\begin{theorem} Let $A$ and $B$ be compact Hopf AF C*-algebras. Let $f_{*}\colon K(A)\longrightarrow K(B)$ be
an isomorphism of ordered K-theory groups that intertwines the K-theory states induced by the co-units and Haar states,  and the  products $\boxproduct_A$ and $\boxproduct_B$. Then $A$ and $B$ are isomorphic or co-anti-isomorphic as Hopf algebras.\label{th:banacheweski.Hopf.for.AF}\label{th:iso.for.compact.AF}
\end{theorem}
 \begin{proof} Using the classification theory for AF algebras\cite{elliott}, we lift the given map to a C*-isomorphism $f\colon A\arrow B.$ The lifted map necessarily intertwines Haar states and co-units when restricted to projections, from which it follows that it intertwines Haar states and co-units in general. By Proposition \ref{prop:box.implies.jordan.for.AF} we then have that the pullback $f^*\colon\Bh\longrightarrow\Ah$ of $f$ is  a Jordan homomorphism at the level of C*-algebras. By Proposition \ref{prop:Hopf.Jordan.maps.preserve.antipodes}, the pullback map is a Jordan *-homomorphism. By Lemma \ref{lem:jordan.auto} the pullback map $f^*$ is then either multiplicative or anti-multiplicative. We thus have, by duality, that $f$ is either an isomorphism or a co-anti-isomorphism of bi-algebras.
 It follows from uniqueness of the Hopf algebra antipode(s) that $f$ is a Hopf algebra (co-anti)isomorphism. 
 \end{proof}
 

\centerline{{\huge {\rotatebox[]{270}{\textdagger}}\!{\rotatebox[]{90}{\textdagger}}}}
\end{document}